\documentclass[12pt]{article}
\usepackage[letterpaper, total={7in, 9in}]{geometry}
\usepackage{amsmath, amsthm, amssymb, bm, mathrsfs}
\usepackage{hyperref, cleveref}

\usepackage[
  backend=biber,
  style=numeric,
  maxnames=99,
  minnames=99,
  giveninits=true,
  url=false, 
  doi=false, 
  isbn=false
]{biblatex}

\renewbibmacro{in:}{}
\DeclareFieldFormat[article]{title}{#1}
\DeclareFieldFormat{issn}{}

\renewbibmacro*{volume+number+eid}{%
  \printfield{volume}%
  \setunit*{\addperiod}%
  \printfield{number}%
}

\renewbibmacro*{issue+date}{%
  \setunit{\addcomma\addspace}%
  \printfield{year}%
  \setunit{\addcomma\addspace}%
  \newunit%
}

\usepackage{hyperref}
\theoremstyle{definition}
\newtheorem{definition}{Definition}[section]

\newtheorem{question}[definition]{Question}

\theoremstyle{plain}
\newtheorem{theorem}[definition]{Theorem}
\newtheorem{lemma}[definition]{Lemma}
\newtheorem{proposition}[definition]{Proposition}

\newtheorem{claim}[definition]{Claim}

\newtheorem{corollary}[definition]{Corollary}
\newcommand{\ra}{\rightarrow}

\usepackage{xcolor}   % in the preamble

\usepackage{amsmath, amssymb}
\usepackage{booktabs}
\usepackage{array}
\usepackage{multirow}

\renewcommand{\le}{\leqslant}
\renewcommand{\ge}{\geqslant}
\renewcommand{\leq}{\leqslant}
\renewcommand{\geq}{\geqslant}

\usepackage{todonotes}

\def \F {\mathbb{F}}

\newcommand{\wt}{\mathrm{wt}}
\newcommand{\diag}{\mathrm{diag}}
\newcommand{\supp}{\mathrm{supp}}
\usepackage[shortlabels]{enumitem}

\providecommand{\keywords}[1]
{
	\small	
	\textbf{\textit{Keywords: }} #1
}

\title{Forbidding Exactly One Hamming Distance}

\author{J\'ozsef Balogh\footnotemark[1]\thanks{Department of Mathematics, University of Illinois Urbana-Champaign, Urbana, IL, USA, and Extremal Combinatorics and Probability Group (ECOPRO), Institute for Basic Science (IBS-R029-C4), Daejeon, South Korea. Partially supported by NSF grants RTG DMS-1937241, FRG DMS-2152488, (UIUC Campus Research Board Award RB26026), Simons Collaboration grant [SFI-MPS-TSM-00013107, JB].
E-mail: \texttt{jobal@illinois.edu}.}
\and Ce Chen\footnotemark[2]
\and Bowen Li
\thanks{Department of Mathematics, University of Illinois Urbana-Champaign, Urbana, IL, USA. Partially supported by NSF grant RTG DMS-1937241 and UIUC Campus Research Board RB 24012.
E-mail: \texttt{\{cechen4, bowenl6\}@illinois.edu}.}
}

\date{}

\begin{document}

\maketitle

\begin{abstract}
% We consider the following generalization of hypercube. The vertex
% Define the $r$-distance graph $H_r(n)$ on the boolean cube $\{0,1\}^n$ that two vertices are adjacent if theie hamming distance is exactly $r$. We determined the size $s$-independent number of the $r$-distance graph $\alpha_s(H_r(n))$ up to asymptotic: we proved that $\alpha_s(H_r(n)) = \Theta(2^n/n^{r/2})$. We used BCH code. 
Addressing questions raised in recent papers, we study the $r$-distance graph $H_r(n)$ on the Boolean cube $\{0,1\}^n$, where two vertices are adjacent if their Hamming distance is exactly $r$. For fixed integers $s \ge 2$ and even $r \ge 2$, we determine the asymptotic order of the $s$-independence number $\alpha_s(H_r(n))$, showing that
\[
\alpha_s\!\left(H_r(n)\right) = \Theta\!\left(\frac{2^n}{n^{r/2}}\right).
\]
The upper bound is derived via a reduction to extremal problems for sunflower-free set systems, while the lower bound is obtained using algebraic constructions based on BCH codes and constant-weight codes.
\end{abstract}


\textbf{\textit{Primary: }}05D05; \textbf{\textit{Secondary:}} 05D40, 94B65, 05C35.

\section{Introduction}
For a fixed integer $r\ge 1$, define the \emph{$r$-distance graph} $H_r(n)$ by
\[
V(H_r(n))=\{0,1\}^n,\qquad
E(H_r(n))=\bigl\{\{x,y\}:d_H(x,y)=r\bigr\}, 
\]
where  $d_H(x,y):=\bigl|\{i:x_i\neq y_i\}\bigr|$ denotes the \textit{Hamming distance}. 
Note that $H_1(n)$ is the usual Hamming graph.

Let $a_3(G)$ be the maximum size of a subset of $V(G)$ which spans a triangle-free graph in $G$. Note, when $r$ is odd, then $\alpha_3(H_r(n)) = 2^n$ as $H_r(n)$ is a bipartite graph. Hence from now on, we assume $r$ is always even and write $r = 2t.$

Castro-Silva, de Oliveira Filho, Slot, and Vallentin~\cite{CastroSilva2023RecursiveTheta} asked the following question in Section 7 of~\cite{CastroSilva2023RecursiveTheta}, which is related to the minima of Krawchuok and Hahn polynomials.  

\begin{question}\label{question triangle}
    What is $\alpha_3(H_r(n))$ for every $n$ and even $r$?
\end{question}

Afterwards, Mukkamala~\cite{mukkamala2025trianglefreesubsetshypercube} 
proved a lower bound of $\Omega(2^n/n^{3r/4})$ for every fixed $r$ using probabilistic method and an upper bound of $O(2^n/n)$ for all even $r$. 
In our note, finding relations to earlier results in coding theory, and on sunflowers, we asymptotically resolve Question~\ref{question triangle} for fixed $r$, and $n$ sufficiently large, by proving the bound $\Theta(2^n/n^{r/2})$. 

As pointed out in~\cite{CastroSilva2023RecursiveTheta}, Question~\ref{question triangle} is also connected to semidefinite optimization. 
Lov\'asz~\cite{lovasz1979shannon}, in his work on the Shannon capacity of the pentagon, 
introduced the theta number, which provides efficiently computable bounds on parameters 
such as the independence number. 
Gr\"otschel, Lov\'asz, and Schrijver~\cite{grotschel1986relaxations} later developed the 
\textit{theta body}, a semidefinite relaxation of the stable set polytope. 
Castro-Silva, de Oliveira Filho, Slot, and Vallentin~\cite{CastroSilva2023RecursiveTheta} 
extended this framework to hypergraphs and used it to derive upper bounds for some problems on the hypercube, 
including Question~\ref{question triangle}.

The lower bound construction in~\cite{mukkamala2025trianglefreesubsetshypercube} is actually an independent set and it was asked whether one can build a genuinely triangle-free set that is not an independent set.
We partially answer this question by exhibiting a triangle-free set that contains many edges, where it matches the upper bound up to lower order term for $r = 2$ and up to a multiplicative constant factor for all $r$. 
% However, as we showed earlier, these constructions are not far from the truth. 
In fact, we will solve the following more general problem. 

Denote $\alpha_s(G)$ (the \emph{$s$-independence number}) the maximum size of a vertex set spanning no copy of $K_s$. These functions were studied earlier in Ramsey-Tur\'an theory~\cite{RT} and in Erd\H{o}s-Rogers types of problems~\cite{erdosrogers}.
Recently, Bohman, Michelen, and Mubayi~\cite{bohman2026largestkrfreesetvertices} determined the $s$-independence number of the random graph $G_{n,1/2}$. In this note, we determine the order of magnitude of the $s$-independence number of $H_r(n)$.

\begin{question}\label{question general s}
   What is $\alpha_s(H_r(n))$, as $n \ra \infty$ for each fixed even constant $r$?
\end{question}

Observe $\alpha(H_r(n))=\alpha_2(H_r(n))$, which is an important function in coding theory. There are at least two major methods on giving upper bounds on $\alpha(H_r(n))$. The first is due to Bassalygo, Cohen, and Z\'emor~\cite{BassalygoCohenZemor2000,CohenZemor1999SubsetSums}. They determined $\alpha(H_r(n))$ up to a multiplicative constant factor, which constant depends on $r$. 

\begin{theorem}[Proposition 6 of~\cite{BassalygoCohenZemor2000}; Proposition 4.4 of~\cite{CohenZemor1999SubsetSums}]\label{thm cohenzemor}
For every fixed $r=2t$,
\[
\alpha(H_{2t}(n))=\Theta\!\left(\frac{2^n}{n^{t}}\right).
\]
\end{theorem}
The lower bound is obtained using results on the BCH code, see, e.g.,~\cite{MacWilliamsSloane1977}, or the Appendix of our note.
The upper bound on $\alpha(H_{2t}(n))$ is obtained by using a result on a famous problem of Erd\H{o}s and S\'os~\cite{erdHos1975problems} on forbidden intersections, see, e.g.,~\cite{ellis2022intersection,Cherkashin2024, linz2026setsystemscontainingsingleton}.
They asked: for $n, k, \ell \in \mathbb{N}$, what is the maximum size $m(n, k, \ell)$ of a family $\mathcal{F} \subset\binom{[n]}{k}$, where $[n]=\{1,\ldots, n\}$, such that
$|A \cap B| \neq  \ell$ for any $A, B \in \mathcal{F}$? Note every $A,B\in\binom{[n]}{k}$ satisfy $d_H(\mathbf{1}_A,\mathbf{1}_B)=2\bigl(k-|A\cap B|\bigr)$. 

The other approach for upper bounding $\alpha(H_{2t}(n))$ is usually called as Delsarte linear programming bound, relating to the problem of Hamming associate scheme, see~\cite{Delsarte1973,BannaiIto1984,DelsarteLevenshtein1998} for more details. 

The rest of the note is organized as follows. 
In Section~\ref{sec Newresult}, we summarize our results. In Section~\ref{Preliminary}, we introduce the background and results used in our proofs. In Section~\ref{sec lower bounds}, we provide the constructions for lower bounds on $\alpha_s(H_r(n))$. In Section~\ref{sec: proof of thm 1.4 and 1.5}, we determine $a_s(H_r(n))$ asymptotically and prove a more explicit upper bound on $\alpha(H_r(n))$ than those in Theorem \ref{thm cohenzemor}. In the Appendix, we include a description of the BCH code construction.

\subsection{New Results}\label{sec Newresult}

We asymptotically solve Question \ref{question general s} for every fixed $r$.
\begin{theorem}\label{thm alpha hrn asymptotic}
     For every fixed  $s$ and $r=2t$, we have
    \[
    \alpha_s(H_{2t}(n)) = \Theta\!\left(\frac{2^n}{n^{t}}\right). 
    \]
\end{theorem}
\noindent
Additionally, we determine $\alpha_s(H_2(n))$.

\begin{theorem}\label{thm:asymp_subsequence r=2}
For every $s$, we have
\[
\limsup_{n\to\infty} \frac{\alpha_s(H_2(n))}{2^n/n} = s-1.
\]
\end{theorem}

\noindent We also have the following bounds with more explicit coefficient in the leading term. 

\begin{theorem}\label{theorem lower bound s-independence}
(i) For every fixed $s \geq 2$ and even $r = 2t$, we have
\[
\alpha_s(H_{2t}(n))\ge \left(\frac{s-1}{t+1}+o(1)\right)\cdot \frac{2^n}{n^{t}} \qquad \text{as } n \to \infty.
\]
(ii) For every fixed $k \in [n]$ and even $r = 2t$, along the subsequence $n=2^m-k$ we have
\[
\alpha_s(H_{2t}(n))\ge \left(s-1+o(1)\right)\cdot \frac{2^n}{n^{t}}.
\]
\end{theorem}

Theorem \ref{thm cohenzemor}  provides no explicit constant, while Theorem~\ref{theorem lower bound s-independence} does. Below we also provide some upper bounds as well. 

\begin{theorem}\label{thm prime power}
For every $i\ge 0$ and even $r = 2t$,  if $t+i$ is a prime power, then
\[
\alpha(H_{2t}(n)) \leq \left( \frac{(2t-1+i)!}{(t-1+i)!} + o(1)\right)\cdot \frac{2^n}{n^{t}}.
\]
\end{theorem}

\section{Preliminaries}\label{Preliminary}

\begin{definition}
A family $A_1, \ldots, A_s$ of distinct sets is a \emph{sunflower} if there exists a set $C$, called the \emph{kernel}, such that $C \subseteq A_i$ for every $i$ and the \emph{petals} $A_i \setminus C$ are pairwise disjoint. We denote by $\mathcal{S}_\ell^{(k)}(s)$ the $k$-uniform sunflower with $s$ petals and kernel of size $\ell$. Define the Tur\'an number $\operatorname{ex}(n, \mathcal{H})$ of an $r$-graph $\mathcal{H}$ the maximum number of  edges in an $r$-graph on $n$ vertices, which does not contain a copy of $\mathcal{H}$ as a subhypergraph. 
\end{definition}

The following theorem was proved by F\"uredi~\cite{furedi1983finite}, 
reducing the general case to the two-petal setting, for which bounds 
were obtained by Frankl and F\"uredi~\cite{FranklFuredi1985}. 
The explicit formulation of the theorem was made later explicit 
by Frankl and F\"uredi~\cite{frankl1987exact}.

% The following theorem is proved by F\"uredi~\cite{furedi1983finite} by reducing the bound for $k$ petals to 2 petals and the bound for 2 petals are given by Frankl and F\"uredi\cite{FranklFuredi1985}. However, it is first explicitly contributed to this problem by Frankl and F\"uredi\cite{frankl1987exact}. 

\begin{theorem}[F\"uredi~\cite{furedi1983finite}; Frankl--F\"uredi~\cite{FranklFuredi1985}]
\label{thm frankl turan number of star}
For every fixed $k$, $\ell$, and $s$ with $1 \le \ell \le k-1$,
\[
\operatorname{ex}(n, \mathcal{S}_\ell^{(k)}(s)) = O\bigl(n^{\max\{\ell,\, k-\ell-1\}}\bigr).
\]
\end{theorem}

More recently, Brad\'a\v{c}, Bu\v{c}i\'c, and Sudakov~\cite{bradavc2023turan} proved bounds that also capture the dependence on $k$, when it grows with $n$.
For our purposes, however, it suffices to know the asymptotic behavior in $n$, when all other parameters are fixed.
\medskip

Define $m_s(n,k,\ell)$ to be the maximum size of a $k$-uniform family $\mathcal{F} \subseteq \binom{[n]}{k}$ 
such that there do not exist $s$ distinct sets $F_1,\dots,F_s \in \mathcal{F}$ with $
|F_i \cap F_j| = \ell$  for every $i \ne j$.
The classical parameter introduced by Erd\H{o}s~\cite{erdHos1975problems} corresponds to the case $s=2$, namely $
m(n,k,\ell) := m_2(n,k,\ell)$,
which denotes the maximum size of a $k$-uniform family with no two sets intersecting in exactly $\ell$ elements. Note that the problem of determining $m_s(n,k,0)$ is equivalent to the  Erd\H{o}s Matching Conjecture~\cite{erdos1965problem}. The following observation is immediate.

\begin{lemma} \label{lemma trivial ms to turan}
    $m_s(n,k,\ell) \le \operatorname{ex}(n, \mathcal{S}_\ell^{(k)}(s)).$
\end{lemma}

\noindent We will also need the following result, proved by Frankl and Wilson~\cite{FranklWilson1981IntersectionTheorems}. 
\begin{theorem}[Frankl--Wilson~\cite{FranklWilson1981IntersectionTheorems}]
\label{thm frankl wilson q-1}
Let $q$ be a prime power and $\mathcal{F} \subseteq \binom{[n]}{k}$. 
If $|F \cap F'| \not\equiv k \pmod q$ for all distinct $F, F' \in \mathcal{F}$, 
then $|\mathcal{F}| \le \binom{n}{q-1}$.
\end{theorem}

\medskip
% \subsection{Bassalygo Transfer Lemma}

% \paragraph{Bassalygo Transfer Lemma}
Let $L_k = \{x \in \{0,1\}^n : |x| = k\}$ denote the $k$-th layer of the hypercube. 
By a slight abuse of notation, we also write $L_k =L_k^r$ for the subgraph of $H_r(n)$ induced on this set.
The following transfer argument is usually attributed to Bassalygo and appears in several papers (e.g.,~\cite{Bassalygo1965NewUpperBounds,BassalygoCohenZemor2000}). For completeness, we include its proof.

\begin{lemma}\label{transfer lemma}
    For every $k, s \in [n]$ we have
\[
\alpha_s(H_{2t}(n)) \le \frac{2^n}{\binom{n}{k}} \cdot \alpha_s(L_k)
= \frac{2^n}{\binom{n}{k}} \cdot m_s(n,k,k-t).
\]
\end{lemma}

\begin{proof}
The equality holds by definition, since $\alpha_s(L_k) = m_s(n, k, k-t)$. It remains to prove the inequality. 

Let $\mathcal{F} \subseteq\{0,1\}^n$ be a $K_s$-free set in $H_{2t}(n)$ with $|\mathcal{F}|=\alpha_s\!\left(H_{2t}(n)\right)$. For each $v \in\{0,1\}^n$, define $
\mathcal{F}_k(v)=(\mathcal{F}+v) \cap\binom{[n]}{k},
$ where addition is in $\mathbb{F}_2^n$.
We claim 
\begin{equation}\label{eq: lem 2.5-1}
    \sum_{v \in\{0,1\}^n}\left|\mathcal{F}_k(v)\right|=|\mathcal{F}|\cdot\binom{n}{k},
\end{equation}
since every pair $(x, y) \in \mathcal{F} \times\binom{[n]}{k}$ contributes exactly once to the left handside of~\eqref{eq: lem 2.5-1}, by noticing that for every $(x,y)$ there is a unique $v$ such that $x+v=y$. 

Therefore, for some $v^* \in\{0,1\}^n$,
\begin{equation}\label{eq 11}
    \left|\mathcal{F}_k\!\left(v^*\right)\right| \geq|\mathcal{F}| \cdot \frac{\binom{n}{k}}{2^n} .
\end{equation}
Set $\mathcal{F}^{\prime}=\mathcal{F}_k\!\left(v^*\right) \subseteq\binom{[n]}{k}$. Since translation by $v^*$ preserves the Hamming distance, $\mathcal{F}+v^*$ is $K_s$-free in $H_{2t}(n)$. Hence, its subset $\mathcal{F}^{\prime}$ is $K_s$-free in $L_k$, thus $\left|\mathcal{F}^{\prime}\right| \leq \alpha_s(L_k)$. Combining this with (\ref{eq 11}) gives
\[
\alpha_s\!\left(H_{2t}(n)\right)=|\mathcal{F}| \leq \frac{2^n}{\binom{n}{k}}\left|\mathcal{F}^{\prime}\right| \leq \frac{2^n}{\binom{n}{k}} \alpha_s(L_k).\qedhere
\]
\end{proof}

We will also use a result of Graham and Sloane on constant-weight codes. 

\begin{theorem}[Graham--Sloane~\cite{graham2003lower}] \label{thm constant weight code}
Let $n$ be a prime power and $r = 2t$ be a fixed even integer. Then, for every $k\in[n]$ we have
\[
\chi(L_k) \le n^{t}.
\]
\end{theorem}
\medskip

To extend Theorem \ref{thm constant weight code} to every sufficiently large $n$, we will also use the following classical result on prime gaps.

\begin{theorem}[Baker--Harman--Pintz~\cite{baker2001difference}] \label{thm Baker-Harman-Pintz}
    There is a constant $c<1$ (one may take $c=0.525$) such that for all sufficiently large $x$ there is a prime in the interval $[x,x+x^{c}]$.

\end{theorem}

\section{Lower Bounds Constructions: Proof of Theorem~\ref{theorem lower bound s-independence}}\label{sec lower bounds}

% \subsection{Layered Construction}

Recall that $L_k=\{x\in\{0,1\}^n:\ |x|=k\}$ is the subgraph induced by the $k$-th layer of $H_r(n)$.

\begin{corollary}\label{construction_vandermonde}
For every fixed even integer $r=2t$ we have
\[
\chi(L_k)\le(1+o(1))\,\cdot n^{t},
\qquad  \textrm{as } n\to\infty.
\]
\end{corollary}

\begin{proof}
By Theorem~\ref{thm Baker-Harman-Pintz}, there exists a prime $q \in[n, n+n^c]$ for some $c < 1$. 
Applying Theorem~\ref{thm constant weight code} with this choice of $q$ yields
$
\chi(H_{2t}(n)[L_k])\leq \chi(H_{2t}(q)[L_k]) \le q^{t} = (1+o(1))\,n^{t}.$
\end{proof}

% \begin{proof}
%    This is a direct application of Theorem~\ref{thm constant weight code} and Theorem \ref{thm Baker-Harman-Pintz}. 
% \end{proof}

\begin{lemma}\label{lemma layer to the cube coloring}
For every fixed even integer $r=2t$ we have
    \[
    \chi(H_{2t}(n)) \leq (t+1)\cdot  \max_k \chi(L_k).
    \]
\end{lemma}

\begin{proof}
    It suffices to partition $[n]$ into $t+1$ classes $S_1, \ldots, S_{t+1}$ such that for every $i\neq j$ from the same class, there is no edge between $L_i$ and $L_j$. 
    Define $S_i = \bigcup_{k \equiv 2i \pmod{2t+2}} \{k, k+1 \}$. Clearly, $S_1, \ldots, S_{t+1}$ partitions $[n]$. Additionally, for every $i\neq j$ from the same class, we either have $|i - j| = 1$ or $|i - j| \geq 2t+1$. In both cases, there is no edge between layers $L_i$ and $L_j$. Thus, this partition naturally induces a proper coloring of $H_{2t}(n)$ from proper colorings of $L_k$ for every~$k$. 
\end{proof}

% \subsection{$s$-independent sets}
%We introduce a method to obtain a $K_s$-free set from an independent set in the Hamming-like graph. 
\begin{lemma} \label{lemma boosting to s-1 parts}
Let $G$ be a graph on $V=\{0,1\}^n$ such that adjacency depends on the Hamming distance, i.e., $x y \in E(G)$ if and only if $d_H(x, y) \in D$ for some $D \subseteq\{1, \ldots, n\}$.
Let $I \subseteq V$ be an independent set in $G$. For every integer $s \geq 3$,
\[
\alpha_s(G) \geq (s-1) \cdot |I| \cdot \left(1 - \frac{(s-2)}{2}\cdot\frac{|I|}{2^n}\right).
\]
In particular, $\alpha_s(G) \geq (s-1 + o(1))\cdot \alpha(G)$ if $\alpha(G) = o(2^n)$.
\end{lemma}

\begin{proof}
Let $u_1,\ldots,u_{s-2}$ be chosen independently and uniformly at random from
$\{0,1\}^n$, and define
\[
S = I \cup (I+u_1) \cup \cdots \cup (I+u_{s-2}).
\]

We first show that $S$ is a $K_s$-free set. 
Since adjacency in $G$ depends only on the Hamming distance, the graph is invariant
under translations of $(\mathbb{Z}/2\mathbb{Z})^n$. Hence, each translate $I+u_k$ is an independent set. The set $S$ is a union of $s-1$ independent
sets, implying that $S$ contains no $K_s$.

We next apply the first moment method
to show that there is a choice of $\{u_1, \ldots, u_{s-2}\}$ such that $S$ has the desired size. 
Set $u_0=0$ so that $I+u_0=I$. By the inclusion--exclusion principle,
\[
|S|
\ge \sum_{k=0}^{s-2}|I+u_k|
   - \sum_{0\le i<j\le s-2}|(I+u_i)\cap(I+u_j)|.
\]
Since $|I+u_k|=|I|$ for every $k$,
\[
\mathbb{E}[|S|]
\ge (s-1)|I|
   - \sum_{0\le i<j\le s-2}\mathbb{E}\bigl[|(I+u_i)\cap(I+u_j)|\bigr].
\]

For fixed $i<j$, we have
$|(I+u_i)\cap(I+u_j)|
  = \bigl|\{(x,y)\in I^2 : x+u_i=y+u_j\}\bigr|.$
The condition $x+u_i=y+u_j$ holds if and only if $u_i+u_j=x+y$.  Since $u_i$ and $u_j$ are chosen
independently and uniformly distributed, $u_i+u_j$ is uniformly distributed on $\{0,1\}^n$, and therefore
$\Pr(u_i+u_j=x+y)=2^{-n}$ for every
$(x,y)\in I^2.$
Summing over all pairs $(x,y)$ gives
\[
\mathbb{E}\bigl[|(I+u_i)\cap(I+u_j)|\bigr]=\frac{|I|^2}{2^n}.
\]
Hence,
\[
\mathbb{E}[|S|]
\ge (s-1)|I|
   - \binom{s-1}{2}\frac{|I|^2}{2^n}
= (s-1)|I|\left(1-\frac{s-2}{2}\cdot\frac{|I|}{2^n}\right).
\]
In particular, there exists a choice of $\{u_1, \ldots, u_{s-2}\}$ for which $|S|$
achieves at least this value. Since $S$ is $K_s$-free, the claimed bound on
$\alpha_s(G)$ follows.
\end{proof}

\begin{proof}[Proof of Theorem~\ref{theorem lower bound s-independence}]
    (i) By Corollary~\ref{construction_vandermonde} and Lemma~\ref{lemma layer to the cube coloring} , there exists a proper coloring of $H_{2t}(n)$ with $(t+1 +o(1)) n^{t}$ color classes. Hence, one can take the largest $s-1$ color classes, whose union is $K_s$-free. \\
(ii) We have two proofs for this.
The first proof treats the BCH code construction as a black box. By the standard BCH code construction (see, e.g.,~\cite{MacWilliamsSloane1977}), there exists an independent set in $H_{2t}(n)$ with size $(1 + o(1))2^n/n^{t}$. Then, we can apply Lemma~\ref{lemma boosting to s-1 parts} to find a $K_s$-free set with size $(s-1 +o(1))2^n/n^{t}$.

For the second proof, we refer the reader to the Appendix, where we give a construction with additional structures. 
By Proposition~\ref{prop:chi_general_ceiling}, there is a proper coloring of $H_r(n)$ with $(1+ o(1)) n^{t}$ color classes. Hence, one can simply take the largest $s-1$ color classes to obtain a $K_s$-free set.
\end{proof}

\section{Proofs of Theorems~\ref{thm alpha hrn asymptotic}, ~\ref{thm:asymp_subsequence r=2} and~\ref{thm prime power}}\label{sec: proof of thm 1.4 and 1.5}

\begin{proof}[Proof of Theorem~\ref{thm alpha hrn asymptotic}]
    
We prove that for each fixed even $r=2t$, 
    \[
    \alpha_s(H_{2t}(n)) = \Theta\left(\frac{2^n}{n^{t}}\right). 
    \]
We first prove the upper bound.
Note that a sunflower with $s$ petals is a special type of a $K_s$ in $H_{2t}(n)$; hence an upper bound on the Tur\'an number of the sunflower implies an upper bound on $\alpha_s$.

For $k = 2t-1$, by Lemma~\ref{transfer lemma}, we have 

\begin{equation}\label{eq: upper bound}
    \alpha_s(H_{2t}(n)) \leq \frac{2^n}{\binom{n}{2t-1}} \cdot m_s(n,\, 2t-1,\, t-1).
\end{equation}
By Lemma~\ref{lemma trivial ms to turan}, and Theorem~\ref{thm frankl turan number of star} applied with 
intersection parameter $\ell = t- 1$ and uniformity $2t - 1$,
\[
m_s(n,\, 2t-1,\, t-1) 
\;\leq\; \mathrm{ex}\!\left(n,\, \mathcal{S}_{t-1}^{(2t-1)}(s)\right) 
\;=\; O\!\left(n^{t-1}\right).
\]
Since $\binom{n}{2t-1} = \Theta(n^{2t-1})$, combining these bounds gives
\[
\alpha_s(H_{2t}(n)) \;\leq\; 2^n \cdot \frac{O(n^{t-1})}{\Theta(n^{2t-1})} \;=\; O\!\left(\frac{2^n}{n^{t}}\right).
\]

The lower bound follows from part~(i) of Theorem~\ref{theorem lower bound s-independence}. 
\end{proof}

%\begin{lemma}\label{lemma alpha_s_upper_bound}
%    For every fixed $s\geq 2$, we have $\alpha_s(H_2(n)) \leq \frac{s-1}{n} \cdot 2^n. $
%\end{lemma}
%\begin{proof}
%Note that $m_s(n,1,0) = s-1$. Applying Lemma~\ref{transfer lemma} with $k = 1$, we get 
%\[
%\alpha_s(H_2(n)) \leq m_s(n,1,0) \cdot \frac{2^n}{n} = \frac{s-1}{n} \cdot 2^n,
%\]
%as desired. 
%\end{proof}

\begin{proof}[Proof of Theorem~\ref{thm:asymp_subsequence r=2}]
    We prove that \[
\limsup_{n\to\infty} \frac{\alpha_s(H_2(n))}{2^n/n} = s-1.
\]
For the upper bound, note that $m_s(n,1,0) = s-1$. Taking $t=1$ in~\eqref{eq: upper bound}, we get 
\[
\alpha_s(H_2(n)) \leq m_s(n,1,0) \cdot \frac{2^n}{n} = \frac{s-1}{n} \cdot 2^n,
\]
as desired. The lower bound follows from part~(ii) of Theorem~\ref{theorem lower bound s-independence}. 
\end{proof}

\begin{proof}[Proof of Theorem~\ref{thm prime power}]

Setting $k = 2t-1 +i$,  by Lemma~\ref{transfer lemma}, we have
\begin{equation}\label{eq ind1}
    \alpha(H_{2t}(n)) \leq \frac{2^n}{\binom{n}{2t-1+i}}\cdot  m_2(n, 2t-1+i, t- 1 +i).
\end{equation}
Since $t +i$ is a prime power, applying Theorem~\ref{thm frankl wilson q-1} with $q = t+i$, we have 
\begin{equation}\label{eq ind2}
    m_2(n, 2t-1+i, t -1 +i) \leq \binom{n}{t+i-1}.
\end{equation}
Combining (\ref{eq ind1}) and (\ref{eq ind2}) together, we have 
\[
\alpha\left(H_{2t}(n)\right) \leqslant\left(\frac{(2t-1+i)!}{(t-1+i)!}+o(1)\right) \cdot \frac{2^n}{n^{t}} . \qedhere
\]
\end{proof}

\medskip

\noindent
\textbf{Acknowledgments:} We would like to thank Haoran Luo, Dadong Peng, and Zolt\'an F\"uredi for helpful discussions.

% \bibliographystyle{abbrv}
% \bibliography{bib}

\printbibliography

\section{Appendix: BCH Codes}
%\subsection{BCH Codes}\label{sec:BCH-codes}
The description for this BCH code construction is not easy to locate, thus we include a proof here for completeness.
One classical version of the BCH code corresponds to a binary linear code of length $N$ with minimum Hamming distance $2t+1$ and dimension at least $N-t \cdot \lceil \log_2 N\rceil.$
The following  is a slight modification of a classical BCH code construction; see, e.g.,~\cite{MacWilliamsSloane1977}.
We remark that Linial, Meshulam, and Tarsi~\cite{linial1988matroidal} obtained the same statement when $r=2$, and Skupień~\cite{SKUPIEN2007990} proved a closely related result with slightly different parameters but without the additional remark on the sizes of the color classes. 

\begin{proposition}\label{prop:chi_general_ceiling}
For every even integer $r = 2t$ and every $n\in\mathbb{N}$,
\[
\chi\left(H_{2t}(n)\right)\le 2^{t\lceil \log_2(n)\rceil}.
\]
Moreover, if $n$ is a power of $2$ or one less than a power of $2$, then there exists a proper coloring achieving the above bound in which every color class has the same size.
\end{proposition}

\begin{proof}
Set
\[
m:=\lceil \log_2(n)\rceil,\qquad N:=2^m,\qquad k:=N-n.
\]
Let $\F:=\F_{2^m}$ be the field of size $2^m=N$, and fix a labeling
\[
\F=\{\gamma_1,\gamma_2,\dots,\gamma_N\}.
\]
We view the $N$-dimensional cube $Q_N=\{0,1\}^N$ as having coordinates
indexed by $\gamma_1,\dots,\gamma_N$.

For $x=(x_1,\dots,x_N)\in\{0,1\}^N$ and an integer $j\ge 1$, define the power sum
\[
S_j(x)\ :=\ \sum_{i=1}^N x_i\,\gamma_i^{\,j}\ \in\ \F.
\]
Define an $\F_2$-linear map $\Phi:\F_2^N\to \F^t$ by
\[
\Phi(x)\ :=\ \bigl(S_1(x),S_3(x),\dots,S_{2t-1}(x)\bigr).
\]
For each $b\in\F^t$, let
\[
C_b\ :=\ \{x\in\F_2^N:\ \Phi(x)=b\}.
\]
These fibers partition $\F_2^N$ into exactly $|\F^t|=N^t$ parts.

\begin{claim}\label{clm: apdx} Every fiber is an independent set in $H_r(N)$.
\end{claim}

\begin{proof}
Take distinct $x,x'\in C_b$ and set $y:=x\oplus x'\in\F_2^N$.
By $\F_2$-linearity, $\Phi(y)=\Phi(x)-\Phi(x')=0$, hence $y\in C_0$ and $y\neq 0$.

In characteristic $2$ we have for every $j\ge 1$,
\begin{equation}\label{eq:doubling_identity_full}
S_{2j}(y)\ =\ S_j(y)^2,
\end{equation}
as $(u+v)^2=u^2+v^2$ and $y_i^2=y_i$ for $y_i\in\{0,1\}$.
Since $y\in C_0$, we have $S_{2\ell-1}(y)=0$ for every $\ell\in[t]$. Iterating~\eqref{eq:doubling_identity_full} shows
\begin{equation}\label{eq:all_moments_full}
S_\ell(y)=0\qquad\text{for every }\ell\in [2t].
\end{equation}

Denote by $\wt(y)$ the number of $1$'s in $y$. We now show that $\wt(y)\ge 2t+1$. Suppose for a contradiction that $w:=\wt(y)\le 2t$.
Let $\supp(y)=\{i_1,\dots,i_w\}$ and $\beta_s:=\gamma_{i_s}$, which are pairwise distinct.

\medskip
\noindent\emph{Case 1: None of the $\beta_s$'s is $0$.}
Then,
\[
\sum_{s=1}^w \beta_s^{\,\ell}=0
\]
for every $\ell\in[w]$ by~\eqref{eq:all_moments_full}.
Define the $w\times w$ matrix $M=\bigl[\beta_s^{\,\ell}\bigr]_{\ell}^{s}$, then $M\mathbf{1}=\mathbf{0}$, where $\mathbf{1}=(1,\dots,1)^T\in\F^w$.
Factoring $\beta_s$ out of column $s$, we obtain
\[
M=\diag(\beta_1,\dots,\beta_w)\cdot
\begin{pmatrix}
1&1&\cdots&1\\
\beta_1&\beta_2&\cdots&\beta_w\\
\vdots&\vdots&\ddots&\vdots\\
\beta_1^{w-1}&\beta_2^{w-1}&\cdots&\beta_w^{w-1}
\end{pmatrix}.
\]
The second factor is the Vandermonde matrix, whose determinant is
$\prod_{1\le i<j\le w}(\beta_j-\beta_i)$, which is nonzero because the $\beta_s$'s are distinct.
Additionally, we have $\prod_{s=1}^w\beta_s\neq 0$ since $\beta_s\in\F^\times$ for every~$s$.
Hence, $\det(M)\neq 0$, implying that $M$ is invertible. This contradicts $M\mathbf{1}=\mathbf{0}$ with $\mathbf{1}\neq \mathbf{0}$.
% Thus this case cannot occur.

\medskip
\noindent\emph{Case 2: One of the $\beta_s$'s equals $0$.}
Since the $\beta_s$'s are distinct, exactly one of them is $0$. Relabel them so that $\beta_w=0$.
For every $\ell\ge 1$ we have $\beta_w^{\,\ell}=0$, so~\eqref{eq:all_moments_full} implies
\[
\sum_{s=1}^{w-1} \beta_s^{\,\ell}=0\qquad \text{for every }\ell\in[w-1].
\]
Since $\beta_1,\dots,\beta_{w-1}$ are distinct and nonzero, applying the same Vandermonde argument to the
$(w-1)\times(w-1)$ matrix $[\beta_s^{\,\ell}]_{\ell}^{s}$ yields a contradiction.
% Thus this case also cannot occur.

In either case we contradict $w\le 2t$, so $\wt(y)\ge 2t+1$.
Hence, $d(x,x')=\wt(x\oplus x')=\wt(y)\ge 2t+1$, implying that $x$ and $x'$ are non-adjacent in $H_r(N)$.
This proves Claim~\ref{clm: apdx}.
\end{proof}

By Claim~\ref{clm: apdx}, $\{C_b\}_{b\in\F^t}$ gives a proper coloring of $H_r(N)$ using at most $N^t$ colors. Identify $\{0,1\}^n$ with the induced subcube
$
U\ :=\ \{(u,0^k)\in\{0,1\}^{n}\times\{0,1\}^{k}\}\ \subseteq\ \{0,1\}^{N}$.
The induced subgraph of $H_r(N)$ on $U$ is isomorphic to $H_r(n)$.
Restricting the above coloring to $U$ yields a proper coloring of $H_r(n)$ with at most $N^t$ colors. Hence,
\[
\chi\!\bigl(H_r(n)\bigr)\ \le\ N^t\ =\ 
% (2^m)^t\ 
% %\ =\ 2^{t\lceil \log_2(n)\rceil}
% \ =\ 
2^{\lceil \log_2(n)\rceil\cdot t},
\]
proving the first part of Proposition~\ref{prop:chi_general_ceiling}.

Now we prove the remark about the size of color classes. Assume $n=2^m$, then $N=n$ and the coloring is defined on $\F_2^n$ by the fibers of the linear map
$\Phi:\F_2^n\to\F^t$. Since $\Phi$ is $\F_2$-linear, every nonempty fiber is an affine translate of $\ker(\Phi)$, thus has exactly the same size. Therefore, the color classes are equal.

For $n=2^m-1$, set $N:=n+1=2^m$ and assume $\gamma_1=0$ without loss of generality.
Let $U:=\{x\in\{0,1\}^N:\ x_1=0\}\cong \{0,1\}^n.$
Because $\gamma_1=0$, the coordinate $x_1$ contributes nothing to $S_j(x)$ for every $j\ge 1$, so $\Phi(x)$ does not depend on $x_1$.
Consequently, for every nonempty fiber $C_b$ we have
$|C_b\cap U|=|C_b|/2$,
thus restricting the fiber coloring to $U$ shrinks every color class by the same factor $2$.
In particular the resulting color classes of $H_r(n)$ have the same size.
\end{proof}

\end{document}